\documentclass[11pt]{amsart}
\usepackage[margin=1in]{geometry}
\usepackage{graphicx}
\usepackage{amssymb,amsmath}
\usepackage{epstopdf}
\usepackage{enumerate}
\usepackage[title]{appendix}
\newcommand{\lsim}{\raisebox{-0.13cm}{~\shortstack{$<$ \\[-0.07cm]
      $\sim$}}~}
\newcommand{\gsim}{\raisebox{-0.13cm}{~\shortstack{$>$ \\[-0.07cm]
      $\sim$}}~}
\newtheorem{lemma}{Lemma}
\newtheorem{theorem}{Theorem}
\newtheorem{corollary}{Corollary}
\begin{document}
\title{Non-uniform Dependence for Euler Equations \\ in Besov Spaces}
\author{Jos\'e Pastrana}
\address{255 Hurley, Department of Mathematics, University of Notre Dame, Notre Dame, Indiana 46556-4618, USA}
\email{jpastran@nd.edu}
\begin{abstract}
We prove the non-uniform continuity of the data-to-solution map of the incompressible Euler equations in Besov spaces, where the parameters $p, q$ and $s$ considered here are such that the local existence and uniqueness result holds.
\end{abstract}
\maketitle
\section{Introduction}
We consider the Cauchy problem governing the motion of a non-viscous and incompressible fluid in a domain $\Omega \subseteq \mathbb{R}^d, d\ge 2$
\begin{equation}\label{eqn:EE}
        \begin{aligned}
        	\partial_t u  + (u \cdot \nabla)u +\nabla p & = 0,\\
		\text{div}\; u & = 0. \\
		u(0,x) & = u_0(x), \; x \in \Omega ,
        \end{aligned}
\end{equation}
where $u:\mathbb{R}\times \Omega \to \mathbb{R}^d$ is the velocity field, $p:\mathbb{R}\times \Omega \to \mathbb{R}$ is the pressure function and $u_0:\Omega \to \mathbb{R}^d$ is the divergence free initial velocity. 

A Cauchy problem is said to be well-posed in the sense of Hadamard if given an initial data one can show the existence of a unique solution that depends continuously on the initial data. Our main concern here is with the local, in time, theory. We first generalize the periodic result of Himonas and Misio\l{}ek \cite{HM} and show that continuous dependence on the initial data in the Eulerian coordinates is the best result one can expect for Besov spaces on $\mathbb{T}^d$.
\begin{theorem}\label{main1}
Let $s \in \mathbb{R}$, $1\le p,q \le \infty$ and $d=2, 3$. The data-to-solution map of the incompressible Euler equations \eqref{eqn:EE} is not uniformly continuous from the unit ball in $B_{p,q}^{s}(\mathbb{T}^d)$ into $C \left( [0,T] , B_{p,q}^{s}(\mathbb{T}^d,\mathbb{R}^d) \right)$.
\end{theorem}

Liu and Tang \cite{LT} extended the periodic result in \cite{HM} to $B_{2,r}^{s}(\mathbb{T}^d)$ for $1\le r \le \infty$. We further extend their result to include the cases $p\not=2$ and use similar scales as in \cite{LT} for the non-periodic case. Recently, in the special case $p=q=\infty$, Misio\l{}ek and Yoneda \cite{MY} have shown that the solution map for the Euler equations is not even continuous in the space of H\"older continuous functions, and thus not locally Hadamard well-posed in $B_{\infty,\infty}^{1+\sigma} \simeq C^{1,\sigma}$ unless restricted to the little H\"older subspace $c^{1,\sigma}$.  More precisely, they showed that the incompressible Euler equations \eqref{eqn:EE} are locally well-posed in the sense of Hadamard in $c^{1,\sigma}(\mathbb{R}^d)$ for any $0 < \sigma < 1$ and $d=2, 3$, where $c^{1,\sigma}$ is the closure of the $C^{\infty}$ functions with respect to the  H\"older norm, see \cite{BL}.

As a consequence to Theorem \ref{main1} we have the following result.
\begin{corollary}\label{main2}
The data-to-solution map of the incompressible Euler equations \eqref{eqn:EE} is not uniformly continuous from the unit ball in $c^{1,\sigma}(\mathbb{T}^d)$ into $C \left( [0,T] , c^{1,\sigma}(\mathbb{T}^d,\mathbb{R}^d) \right)$ for any $0 < \sigma <1$ and $d=2, 3$.
\end{corollary}

For the non-periodic case we have the following result.
\begin{theorem}\label{main3}
Let $s>1 + d/2$, $2 \le q \le \infty$ and $d=2, 3$. The data-to-solution map of the incompressible Euler equations \eqref{eqn:EE} is not uniformly continuous from the unit ball in $B_{2,q}^{s}(\mathbb{R}^d)$ into $C \left( [0,T] , B_{2,q}^{s}(\mathbb{R}^d,\mathbb{R}^d) \right)$.
\end{theorem}
The proof of this result uses estimates for solutions to a linear transport equation.

For background on well-posedness of the Euler equations we refer the reader to the monographs of Bahouri, Chemin and Danchin \cite{BCD}, Bertozzi and Majda \cite{MB} and Chemin \cite{C}. The first rigorous results in this direction were proved in the framework of H\"older spaces by  Gyunter \cite{G}, Lichtenstein \cite{L} and Wolibner \cite{W}, and subsequently by Kato \cite{K2}, Yudovich \cite{Y}, Ebin and Marsden \cite{EM} and others. Concerning the properties of the data-to-solution map, the first results can be found in \cite{EM} and in the work of Kato \cite{K}, who among other things showed that the solution map for Burgers' equation is not H\"older continuous in $H^s$ topology for any H\"older exponent. Further continuity results for the solution map of the Euler equations were obtained in $H^s$ by Kato and Lai \cite{KL}, and subsequently in $W_p^s$ by Kato and Ponce \cite{KP}. Ever since, the subject has become an active area of research involving many non-linear evolution equations.

 It is worth pointing out in passing that Pak and Park \cite{PP} established existence and uniqueness of solutions of the Euler equations in $B_{\infty,1}^{1}$ and showed that the solution map is in fact Lipschitz continuous when viewed as a map between $B_{\infty,1}^{0}$ and  $C ( [0,\infty]; B_{\infty,1}^{0})$. Later, Cheskidov and Shvydkoy \cite{CS} proved that the solution of the Euler equations cannot be continuous as a function of the time variable at $t = 0$ in the spaces $B_{r,\infty}^{s}(\mathbb{T}^d)$ where $s >0$ if $2 < r \le \infty$ and $s > d(2/r-1)$ if $1 \le r \le 2$. Furthermore, Bourgain and Li \cite{BL1,BL2} showed that for $d = 2, 3$ the Euler equations are strongly ill-posed in the Sobolev space $W^{d/p+1,p}$ for any $1 \le p < \infty$, the Besov space $B_{p,q}^{d/p+1}$ for any $1 \le p < \infty$, $1 < q \le \infty$, and in the classical spaces $C^{m}(\mathbb{R}^d)$ and $C^{m-1,1}(\mathbb{R}^d)$ for any integer $m\ge 1$. Most recently, Holmes, Keyfitz and Tiglay \cite{HKT} proved the non-uniform continuity of the solution map for compressible gas in the Sobolev spaces, and Holmes and Tiglay \cite{HT} have obtained  similar results for the Hunter Saxton equation in Besov spaces.
\section{Preliminaries}\label{prelim}
In this section we introduce notation and preliminary results. The solutions that we consider take values in the Besov spaces $B_{p,q}^{s}$. Given a smooth bump function $\varphi_0$ supported in the ball  of radius 2 and equal to one in the ball of radius 1, we let $\varphi_1(\xi) = \varphi_0(\xi/2) - \varphi_0(\xi)$ and set $\varphi_j(\xi) = \varphi_1(\xi/2^{j-1})$ for all $j \in \mathbb{N}$.  Any such dyadic partition of unity defines a frequency restriction operator on functions, $\varphi_j(D)$, via the relation $\mathcal{F}(\varphi_j(D)f)(\xi) := \varphi_j(\xi)\hat f(\xi)$. In what follows we will let $X$ denote either $\mathbb{T}^d$ or $\mathbb{R}^d$ and let $Y'(X)$ denote either the space of linear functionals on $C^\infty(\mathbb{T}^d)$ or the space of tempered distributions $\mathcal{S}'(\mathbb{R}^d)$.

For any $s\in \mathbb{R}$ and $ 1 \le p,q \le \infty$ we define the vector-valued distributions
$$ B_{p,q}^{s}(X):= \left\{f :   f \in Y'(X) \; \;  \text{and} \;\; \|f\|_{ B_{p,q}^{s}(X)}  < \infty  \right\},$$
with norm as in \cite{BCD} and \cite{ST} given by
$$ \|f\|_{B_{p,q}^{s}(X)}:= \begin{cases} 
       \left( \displaystyle\sum_{j = 0}^{\infty} 2^{jsq}\| \varphi_j(D)f \|_{L^p(X)}^q\right)^{1/q}, & \text{if} \; q < \infty \\
     \displaystyle \sup_{j \ge 0} 2^{js}\| \varphi_j(D)f \|_{L^p(X)}, & \text{if} \; q = \infty .\\
   \end{cases}$$
   
We shall use the symbols $ \simeq $ and $\lsim$ to denote estimates that hold up to a universal constant.\\

We will need the following properties of Besov spaces.
\begin{lemma}[See \cite{BCD} and \cite{T83}]\label{Bprops}
Let $s>0$ and $1 \le p, q \le \infty$.

We have the following Moser type inequality:
\begin{equation}\label{eqn:moser}   \| fg\|_{B_{p,q}^{s}(\mathbb{R}^d)} \lsim \|f \|_{L^{p_1}(\mathbb{R}^d)} \; \| g\|_{B_{p_2,q}^{s}(\mathbb{R}^d)} +  \| g\|_{L^{p_3}(\mathbb{R}^d)} \; \|f \|_{B_{p_4,q}^{s}(\mathbb{R}^d)},\end{equation}
where $1 \le p_1, p_2, p_3, p_4 \le \infty$ and $\frac{1}{p} = \frac{1}{p_1} + \frac{1}{p_2} = \frac{1}{p_3} + \frac{1}{p_4}$.

The space $B_{p,q}^{s}(\mathbb{R}^d)\cap L^\infty(\mathbb{R}^d)$ is a Banach algebra. Moreover, $B_{p,q}^{s}(\mathbb{R}^d)$ is a Banach algebra if and only if $B_{p,q}^{s}(\mathbb{R}^d)\hookrightarrow L^\infty(\mathbb{R}^d)$ if and only if $s > d/p$. In particular,
\begin{equation}\label{eqn:algp} \| fg\|_{B_{p,q}^{s}(\mathbb{R}^d)} \lsim
\|f \|_{B_{p,q}^{s}(\mathbb{R}^d)} \; \| g\|_{B_{p,q}^{s}(\mathbb{R}^d)}. \end{equation}

Let $s = \theta s_1 + (1-\theta)s_2$, where $\theta \in [0,1]$. Then we have the following interpolation inequality
\begin{equation}\label{eqn:interp} \| f\|_{B_{p,q}^{s}(\mathbb{R}^d)} \lsim \|f \|_{B_{p,q}^{s_1}\mathbb{R}^d)}^\theta \; \| f\|_{B_{p,q}^{s_2}(\mathbb{R}^d)}^{1-\theta} .
\end{equation}
\end{lemma}

It is known that the pressure term can be removed from the Euler equations, in fact applying the divergence operator to the first equation in \eqref{eqn:EE} and then solving for $p$ gives 
\begin{equation}\label{eqn:hodge1} \nabla p = - \nabla\Delta^{-1}\text{div}\left( (u \cdot \nabla)u \right) = -(1-P)\left( (u \cdot \nabla)u \right),
\end{equation}
where the operator $P = 1 - \nabla\Delta^{-1}\text{div}$ is the $L^2$-orthogonal projection onto the divergence free part in the Hodge's decomposition and satisfies the estimate 
\begin{equation}\label{eqn:hodge2} \| Pf\|_{B_{p,q}^{s}} \le \| f\|_{B_{p,q}^{s}},
\end{equation}
see \cite{T83}. We use \eqref{eqn:hodge1} and the notation $\nabla_{a}b:= (a\cdot \nabla)b$ to deduce the following non-local form of the Cauchy problem \eqref{eqn:EE}
\begin{equation}\label{eqn:NLEE}
        \begin{aligned}
        	\partial_t u  + \nabla_uu - \nabla\Delta^{-1}\text{div}\left( \nabla_uu \right) & = 0,\\
		\text{div}\; u & = 0, \\
		u(0,x) & = u_0(x).
        \end{aligned}
\end{equation}

Local well-posedness of the Euler equations has been established by many authors. We summarize the result, as found in \cite{BCD}, in the form which is convenient for our purposes: if $1\le p,q \le \infty$, $s\in \mathbb{R}$ is such that $s > 1 + d/p$ and $u_0 \in B_{p,q}^{s}$ satisfies $\text{div}\; u_0  = 0 $, then there is a  $T = T(u_0)$, such that \eqref{eqn:EE} has a unique solution $u$ in $C\left([0,T];B_{p,q}^{s}\right)$. Furthermore, we have the estimate
\begin{equation}\label{eqn:energy} \|u(t)\|_{B_{p,q}^{s}} \lsim \|u_0\|_{B_{p,q}^{s}}, \;\;\;\; \text{for} \;\; 0 \le t\le T. \end{equation}

Our general strategy for proving the Theorems \ref{main1}  and \ref{main3} will be the following: consider two sequences of solutions to \eqref{eqn:NLEE}, $u^{+1, n}$ and $u^{-1, n}$, satisfying the following conditions:

\begin{enumerate}[(i)]
\item $u^{+1, n}$ and $u^{-1, n}$ are confined to a ball in $B_{p,q}^{s}(X)$, i.e., for any $t \in [0,1]$ we have
$$ \|u^{\pm1,n}(t,\cdot)\|_{B_{p,q}^{s}(X)} \lsim 1.$$
\item At time $t = 0$, the two sequences are converging as $n \to \infty$, that is
$$ \|u^{+1,n}(0,\cdot) - u^{-1,n}(0,\cdot)\|_{B_{p,q}^{s}(X)} \to 0.$$
\item For any $t \in (0,1]$ they remain far apart, more precisely 
$$ \liminf_{n \to \infty} \|u^{+1,n}(t,\cdot) - u^{-1,n}(t,\cdot)\|_{B_{p,q}^{s}(X)} \gsim \sin t.$$
\end{enumerate}

Dimensions $d=2,3$ represent the physical relevant cases in the study of fluid motion. In what follows we fix the spatial dimension to $d=2$ for simplicity of presentation. 
\section{Non-uniform dependence in Besov spaces - the periodic case}\label{pbesov}
In this section we prove that the solution map is not uniformly continuous in the periodic case. \\

To proceed we will need the following estimates for high frequencies in Besov spaces.
\begin{lemma}\label{tlemp}
Let $s\in \mathbb{R}$ and $1\le p,q \le \infty$. For any constant $a \in \mathbb{R}$ we have the estimates
$$ \|n^{-s}\cos (n \cdot - a) \|_{B_{p,q}^{s}(\mathbb{T})} \simeq 1, \qquad \|n^{-s}\sin (n \cdot - a) \|_{B_{p,q}^{s}(\mathbb{T})} \simeq 1, \quad n \gg 1.$$
\end{lemma}

\begin{proof} We will proceed by cases. 

Case I: $1\le q < \infty$. Using the formula for the inverse Fourier transform we get
\begin{align*}
\|\cos(n\cdot-a) \|_{B_{p,q}^{s}(\mathbb{T})}^q
& = \sum_{j = 0}^{\infty}2^{jsq} \Big\|\sum_{k\in \mathbb{Z}^2} \varphi_j(k) \pi(e^{-ia}\delta_{k}^{n} + e^{ia}\delta_{k}^{-n}) e^{ikx} \Big\|_{L^{p}(\mathbb{T})}^{q}.
\end{align*} 

It is clear that there are only two non-zero terms in the inner sum and thus the above equals
\begin{align*}  &  \pi^q  \sum_{j = 0}^{\infty}2^{jsq} \Big\|\varphi_j(n) e^{-ia}e^{inx} + \varphi_j(-n)e^{ia}e^{i(-n)x} \Big\|_{L^{p}(\mathbb{T})}^{q}.
\end{align*}

Since any dyadic partition induces an equivalent norm we can further assume that the $\varphi_j$'s are radial. We then obtain the equivalent expression
\begin{align*} 
 (2\pi)^q  \sum_{j = 0}^{\infty}2^{jsq} \varphi_j(n)^q \left \| \frac{ e^{-ia}e^{inx} + e^{ia}e^{i(-n)x} }{2} \right\|_{L^{p}(\mathbb{T})}^{q}  & \simeq  \| \cos(n\cdot - a) \|_{L^{p}(\mathbb{T})}^q \sum_{j = 0}^{\infty}2^{jsq}\varphi_j(n)^{q} 
 \\ & =: (*).
\end{align*}

For any $n$, we have $\| \cos(n \cdot ) \|_{L^{p}(\mathbb{T})} + \| \sin (n \cdot ) \|_{L^{p}(\mathbb{T})} \simeq 1$. Since $\varphi_j$ defines a dyadic partition of unity, the terms $\varphi_j(n)$ vanish except for $l \in \mathbb{N}_0$ for which $2^{l-1} \le n \le 2^{l+1}$ and
\begin{align*} 
(*)& \simeq 2^{(l-1)sq}\varphi_{l-1}(n)^{q} + 2^{lsq}\varphi_l(n)^{q}  + 2^{(l+1)sq}\varphi_{l+1}(n)^{q}  \\
& \simeq  n^{sq}.
\end{align*}

Case II: $q = \infty$. In this case we have
\begin{align*} \|\cos(n\cdot-a) \|_{B_{p,\infty}^{s}(\mathbb{T})}  & = \sup_{j \ge 0}2^{js} \Big\|\sum_{k\in \mathbb{Z}^2} \varphi_j(k)\left( \cos (nx-a) \right)\hat\;(k) e^{ikx} \Big\|_{L^{p}(\mathbb{T})} .
\end{align*}

Similar computations gives
\begin{align*} 2^{js} \Big\|\sum_{k\in \mathbb{Z}^2} \varphi_j(k)\left( \cos (nx-a) \right)\hat\;(k) e^{ikx} \Big\|_{L^{p}(\mathbb{T})} & \simeq 2^{js}\| \cos(n\cdot - a) \|_{L^{p}(\mathbb{T})} \varphi_j(n) \\ & \simeq n^s \varphi_j(n).
\end{align*}

By construction all defining functions are uniformly bounded. Again, for each $n \gg 1$ $\varphi_j(n)$ vanish except for $l$ satisfying $2 ^{l-1} \le n \le 2^{l+1}$, furthermore $\varphi_{l-1}(n) + \varphi_{l}(n) + \varphi_{l+1}(n) = 1$. The non-negativity of the $\varphi_l$'s implies that at least one of these three terms must be greater than or equal to $1/3$.  Thus, taking supremum over $j \in \mathbb{N}_0$ gives 
$$ \| \cos(n\cdot-a) \|_{B_{p,\infty}^{s}(\mathbb{T})} \simeq n^{s}.$$
The proof for the sine function is similar.
\end{proof}

\subsection{ \bf Proof of Theorem \ref{main1}}\label{proof1}
We consider two sequences of solutions 
\begin{equation}\label{eqn:ftseq}
u^{\pm 1,n}(t,x_1,x_2) = \left( \frac{\pm1}{n} + \frac{1}{n^s} \cos  \left(n x_2 \mp t \right) \;,\; \frac{\pm1}{n} + \frac{1}{n^s} \cos  \left(n x_1 \mp t \right) \right)
\end{equation}
where $s\in \mathbb{R}$, $(x_1,x_2) \in \mathbb{T}^2$, $n \in \mathbb{N}$, $t\in [0,1]$ and verify that conditions (i)-(iii) are satisfied. This family of solutions to the incompressible Euler equations \eqref{eqn:NLEE} was introduced in \cite{HM} and used there to show that the solution map is not uniformly continuous in Sobolev $H^s$ spaces.

Using Lemma \ref{tlemp} we see that for any $t \in [0,1]$
\begin{align*}
\|u^{\pm 1,n}(t,\cdot)\|_{B_{p,q}^{s}(\mathbb{T}^2)} 
& \lsim \left \| n^{-1}\right\|_{B_{p,q}^{s}(\mathbb{T}^2)} + \left \| n^{-s} \cos  \left(n \cdot \mp  t \right)  \right\|_{B_{p,q}^{s}(\mathbb{T}^2)}  \\ & \lsim 1,
\end{align*}
thus, the family of solutions $u^{\pm 1,n}$ is confined to ball in $B_{p,q}^{s}$, this is (i). 

The angle sum formula gives
\begin{equation}\label{eqn:difftsol}u^{+1,n}(t,x) - u^{-1,n}(t,x)= \left( \frac{2}{n} + \frac{2}{n^{s}}\sin t \sin (n x_2)\;,\; \frac{2}{n} + \frac{2}{n^{s}}\sin t \sin (n x_1) \right),
\end{equation}
from which (ii) follows, in fact
\begin{align*} \| u^{+1,n}(0,\cdot) - u^{-1,n}(0,\cdot)\|_{B_{p,q}^{s}(\mathbb{T}^2)} & \simeq \left \| n^{-1}\right \|_{B_{p,q}^{s} (\mathbb{T}^2)} \simeq n^{-1} \to 0, \quad\text{as}\quad n\to\infty.
\end{align*}

Finally, using \eqref{eqn:difftsol} and the triangle inequality we get
\begin{align*}
\| u^{+1,n}(t,\cdot) - u^{-1,n}(t,\cdot)\|_{B_{p,q}^{s}(\mathbb{T}^2)} & \simeq   \left \| n^{-1} + n^{-s} \sin t \sin (n \cdot)  \right\|_{B_{p,q}^{s}(\mathbb{T}^2)} \\
& \gsim  |\sin t| \; \|n^{-s} \sin (n \cdot)  \|_{B_{p,q}^{s} (\mathbb{T}^2)} - \left \| n^{-1}\right\|_{B_{p,q}^{s}(\mathbb{T}^2)} \\
&  \simeq  |\sin t| - n^{-1}.
\end{align*}

In the last line Lemma \ref{tlemp} was applied to the sine function. Taking the limit as $n\to \infty$ we get
$$ \liminf_{n \to \infty} \|u^{+1,n}(t,\cdot) - u^{-1,n}(t,\cdot)\|_{B_{p,q}^{s}(\mathbb{T}^2)} \gsim \sin t  \;,\;\;\; \forall t \in (0,1],$$

which completes the proof in the case $d= 2$.\\

For the case $d=3$, it suffices to consider the following vector-valued solution to the Euler equations
$$ u^{\pm 1,n}(t,x_1,x_2,x_3) = \left( \pm n^{-1} + n^{-s} \cos  \left(n x_2 \mp t \right) \;,\; \pm n^{-1} + n^{-s} \cos  \left(n x_1 \mp t \right),0 \right),$$
for which Lemma \ref{tlemp} and the non-uniform dependence argument apply with appropriate modification.
\section{Non-uniform dependence in H\"older spaces - the periodic case}\label{pholder}

In this section we prove Corollary \ref{main2}.  Recall that Misio\l{}ek and Yoneda \cite{MY} proved local well-posedness in the sense of Hadamard in the space $c^{1,\sigma}$ of functions in $C^{1,\sigma}$ satisfying the vanishing condition
\begin{equation}\label{eqn:vcond} \lim_{h \to 0} \sup_{0 < |x-y| < h}\frac{|\partial^\alpha u(x) - \partial^\alpha u(y)|}{|x-y|^\sigma} = 0\;\;, \forall |\alpha| = 1.
\end{equation}

Besov spaces arise naturally when studying the Fourier analytic characterization of H\"older continuity through the relation $ B_{\infty,\infty}^{s} =  C^{[s],\{s\}}$, where $s = [s] +\{s\}$ is the decomposition of $s$ into its integer and fractional parts. Because the case $p=q=\infty$ is already included in the scope of Theorem \ref{main1} we see that the result holds in $C^{1,\sigma}(\mathbb{T}^d)$ where it is expected for as shown in \cite{MY} the solution map is not even continuous. Thus our result states that in $c^{1,\sigma}(\mathbb{T}^d)$, continuity is optimal. Because $c^{1,\sigma}$ inherits its norms from $C^{1,\sigma}$, we see that applying the previous result to $p=q=\infty$ will give the result for $c^{1,\sigma}$ once we prove that the family of solutions in \eqref{eqn:ftseq} satisfies \eqref{eqn:vcond}. 

\subsection{ \bf Proof of Corollary \ref{main2} }\label{proof2}
Upon relying on a symmetric argument we let $u = u^{+1,n}$ and $\alpha  = e_1 = (1,0)$ be a standard basis element of $\mathbb{R}^2$. The computations for $u = u^{-1,n}$ or $\alpha = e_2 = (0,1)$ are similar. We note that
\begin{align*}
 \partial^\alpha u(x) - \partial^\alpha u(y) & = \partial_{x_1} \left( \frac{1}{n} + \frac{1}{n^{1 + \sigma }} \cos  \left(n x_2 - t \right) \;,\; \frac{1}{n} + \frac{1}{n^{1 + \sigma }} \cos  \left(n x_1 - t \right) \right)\\
& - \partial_{y_1} \left( \frac{1}{n} + \frac{1}{n^{1 + \sigma }} \cos  \left(n y_2 - t \right) \;,\; \frac{1}{n} + \frac{1}{n^{1 + \sigma }} \cos  \left(n y_1 - t \right) \right)\\
& = \left( 0 \;,\; -\frac{1}{n^{\sigma}} \sin  \left(n x_1 - t \right) + \frac{1}{n^{\sigma}} \sin  \left(n y_1 - t \right) \right).
\end{align*}

Applying the mean value theorem we see that 
\begin{align*}
|\partial^\alpha u(x) - \partial^\alpha u(y)| & = \frac{1}{n^{\sigma}} \left| \sin  \left(n x_1 - t \right) - \sin  \left(n y_2 - t \right) \right| \\
& \le  \frac{n}{n^{\sigma}} \cdot 1 \cdot |x_1 - y_1| \\
& \le n^{1-\sigma} |x - y|,
\end{align*}
thus, for $|x-y| \not= 0$  we get
$$\frac{|\partial^\alpha u(x) - \partial^\alpha u(y)|}{|x-y|^\sigma} \le n^{1-\sigma} |x - y|^{1-\sigma}.  $$

Because $\sigma \in (0,1)$, taking sup over all $0 < |x-y| < h$ gives
$$ \sup_{0 < |x-y| < h}\frac{|\partial^\alpha u(x) - \partial^\alpha u(y)|}{|x-y|^\sigma} \le n^{1-\sigma} h^{1-\sigma},  $$
so that taking the limit we obtian
$$  \lim_{h \to 0} \sup_{0 < |x-y| < h}\frac{|\partial^\alpha u(x) - \partial^\alpha u(y)|}{|x-y|^\sigma} \le n^{1-\sigma} \lim_{h \to 0} h^{1-\sigma}= 0. $$

This completes the proof of Corollary \ref{main2}.

\section{Non-uniform dependence in Besov spaces - the non-periodic case}\label{npbesov}

Just as in Section \ref{pbesov}, our general strategy here will be to find two sequences of solutions to \eqref{eqn:NLEE} for which the conditions (i)-(iii) are satisfied. We make use of the approximate solutions technique as in \cite{HM}. For more details about this technique refer to Tzvetkov \cite{T}. We first select two sequences of bounded approximate solutions, which are arbitrarily close at time zero but separated at later times. We then show that the difference between the approximate solutions and the exact solutions to equation \eqref{eqn:NLEE} that they induce is negligible in the Besov norm. Finally, we show that the exact solutions remain separated from each other at later times using triangle inequality and the fact that the approximate solutions are separated at later times but still converge to the exact solutions. 

To proceed we will need the following estimates.
\begin{lemma}\label{tlemnp}
 Let $\delta > 0$, $\sigma > 0$ and $1 \le q \le \infty$. For any Schwartz function $f \in \mathcal{S}(\mathbb{R})$ we have  
\begin{equation}\label{eqn:npest1}
\lambda^{\delta/2} \|f \|_{L^2(\mathbb{R})} \lsim 
\left \|f \left(\frac{\cdot}{\lambda^\delta}\right) \right\|_{B_{2,q}^{\sigma}(\mathbb{R})} \lsim \lambda^{\delta/2}\|f \|_{B_{2,q}^{\sigma}(\mathbb{R})}, \quad \lambda \gg 1.\end{equation}

Furthermore, for $2 \le q \le \infty$ and any constant $a\in \mathbb{R}$ we have the estimate
\begin{equation}\label{eqn:npest2} \left \|f \left(\frac{\cdot}{\lambda^\delta}\right) \cos(\lambda \cdot - a) \right \|_{B_{2,q}^{\sigma}(\mathbb{R})} \simeq \lambda^{\sigma +\delta/2}\|f \|_{L^2(\mathbb{R})}, \quad \lambda \gg 1,\end{equation}
which also holds when $\cos(\lambda \cdot - a)$ is replaced by $\sin(\lambda \cdot - a)$.
\end{lemma}
\begin{proof} 

We use an equivalent formulation of Besov spaces, see \cite{T13}. That is, $f\in B_{p,q}^{\sigma}(\mathbb{R}^d)$ for $0 < \sigma < m$, $m \in \mathbb{N}$,  and $1 \le p \le \infty$ if and only if it has finite norm
$$  \|f\|_{L^{p}(\mathbb{R}^d)} + \begin{cases} 
        \left(\displaystyle\int_{0}^{1}t^{-\sigma q}\sup_{|h|\le t}\|\Delta_{h}^{m}f\|_{L^{p}(\mathbb{R}^d)}^{q} \; \frac{dt}{t}\right)^{1/q}, & \text{if} \; 1 \le q < \infty \\
    \displaystyle  \sup_{h \not =0}|h|^{-\sigma}\|\Delta_{h}^{m} f\|_{L^p(\mathbb{R}^d)},  & \text{if} \; q = \infty .\\
   \end{cases}$$

The iterated differences $\Delta_{h}^{m}$ are defined as follows. 
$$\Delta_{h}^{1}f(x):= f(x+h)-f(x), \quad \Delta_{h}^{m+1}f(x): = \Delta_{h}^{1}(\Delta_{h}^{m}f)(x) \quad m \ge 1.$$ 

We proceed by steps. \\

Step I. We first prove the estimates in \eqref{eqn:npest1}.\\

The first estimate follows from a change of variables. 

$$ \left\|f\left( \frac{\cdot}{\lambda^{\delta}}\right) \right\|_{B_{2,q}^{\sigma}(\mathbb{R})} \gsim \left\|f\left( \frac{\cdot}{\lambda^{\delta}}\right) \right\|_{L^2(\mathbb{R})} = \lambda^{\delta/2} \|f \|_{L^2(\mathbb{R})}.  $$

The definition of the iterated differences and induction gives 
\begin{equation}\label{eqn:itdiff} \Delta_{h}^{m}\left[f\left(\frac{\cdot}{\lambda^\delta}\right)\right](x) = \Delta_{h/\lambda^\delta}^{m}\big[f\big] \left(\frac{x}{\lambda^\delta}\right),  \quad  \Delta_{h}^{m}f(x) = \sum_{k=0}^{m} (-1)^{m-k} {m \choose k} f(x +hk).\end{equation}

For the second estimate we make changes of variables and use \eqref{eqn:itdiff} to get
\begin{align*} \left \| f\left( \frac{\cdot}{\lambda^{\delta}}\right) \right \|_{B_{2,q}^{\sigma}(\mathbb{R})} & \simeq \left \|f\left( \frac{\cdot}{\lambda^{\delta}}\right) \right \|_{L^2(\mathbb{R})} +  \left(\int_{0}^{1}t^{-\sigma q}\sup_{|h|\le t}\left \|\Delta_{h}^{m} \left[f\left( \frac{\cdot}{\lambda^{\delta}}\right) \right](\cdot)  \right \|_{L^{2}(\mathbb{R})}^{q} \; \frac{dt}{t}\right)^{1/q} \\ 
& = \lambda^{\delta/2} \|f \|_{L^2(\mathbb{R})} + \lambda^{-\sigma\delta}\left(\int_{0}^{1}(t/\lambda^\delta)^{- \sigma q}\sup_{|h/\lambda^\delta| \le t/\lambda^\delta}\left \| \Delta_{h/\lambda^\delta}^{m}f \left(\frac{\cdot}{\lambda^\delta}\right)  \right \|_{L^{2}(\mathbb{R})}^{q} \;1/\lambda^\delta \frac{dt}{t/\lambda^\delta} \right)^{1/q}\\
& = \lambda^{\delta/2} \|f \|_{L^2(\mathbb{R})} + \lambda^{-\sigma \delta} \lambda^{\delta/2}\left(\int_{0}^{1/\lambda^\delta}t^{-\sigma q}\sup_{|h| \le t}\left \| \Delta_{h}^{m}f \right \|_{L^{2}(\mathbb{R})}^{q} \; \frac{dt}{t} \right)^{1/q}\\
& \le \lambda^{\delta/2} \| f \|_{B_{2,q}^{\sigma}(\mathbb{R})}.
\end{align*}

Similarly, 
\begin{align*} \left \| f\left( \frac{\cdot}{\lambda^{\delta}}\right) \right \|_{B_{2,\infty}^{\sigma}(\mathbb{R})} & \simeq \left \|f\left( \frac{\cdot}{\lambda^{\delta}}\right) \right \|_{L^2(\mathbb{R})}  + \sup_{h \not =0}|h|^{-\sigma}\left \|\Delta_{h}^{m}  \bigg[ f\left( \frac{\cdot}{\lambda^{\delta}}\right)\bigg] (\cdot) \right\|_{L^2(\mathbb{R})}\\ 
& = \lambda^{\delta/2} \|f \|_{L^2(\mathbb{R})}  + \sup_{h \not =0}|h|^{-\sigma}\left \|\Delta_{h/\lambda^\delta}^{m}  f\left( \frac{\cdot}{\lambda^{\delta}}\right) \right\|_{L^2(\mathbb{R})} \\
& = \lambda^{\delta/2} \|f \|_{L^2(\mathbb{R})}  + \lambda^{\delta/2} \lambda^{-\sigma\delta}\sup_{h/\lambda^\delta \not =0}|h/\lambda^{\delta}|^{-\sigma}\left \|\Delta_{h/\lambda^\delta}^{m}  f \right\|_{L^2(\mathbb{R})} \\
& = \lambda^{\delta/2} \|f \|_{L^2(\mathbb{R})}  + \lambda^{\delta/2} \lambda^{-\sigma\delta}\sup_{h \not =0}|h|^{-\sigma}\left \|\Delta_{h}^{m}  f \right\|_{L^2(\mathbb{R})}  \\
& \le \lambda^{\delta/2} \| f \|_{B_{2,\infty}^{\sigma}(\mathbb{R})}.
\end{align*}

Step II. We now prove the estimates in \eqref{eqn:npest2}.\\

The upper estimate follows from  \cite{HM}, in fact for $2 < q \le \infty$ we have $H^{\sigma} \simeq B_{2,2}^{\sigma}  \hookrightarrow  B_{2,q}^{\sigma}$ and
\begin{align*}
\left \|f \left(\frac{\cdot}{\lambda^\delta}\right) \cos(\lambda \cdot - a) \right \|_{B_{2,q}^{\sigma}(\mathbb{R})} 
& \lsim \left \|f \left(\frac{\cdot}{\lambda^\delta}\right) \cos(\lambda \cdot - a) \right \|_{H^{\sigma}(\mathbb{R})}  \lsim \lambda^{\sigma + \delta/2}\|f \|_{L^2(\mathbb{R})}.
\end{align*}

For the lower estimate we see that for any $1 \le q \le \infty$ we have $B_{2,q}^{s}  \hookrightarrow  B_{2,\infty}^{s}$ which gives
$$\left \|f \left(\frac{\cdot}{\lambda^\delta}\right) \cos(\lambda \cdot - a) \right \|_{B_{2,q}^{\sigma}(\mathbb{R})}  \ge \left \|f \left(\frac{\cdot}{\lambda^\delta}\right) \cos(\lambda \cdot - a) \right \|_{B_{2,\infty}^{\sigma}(\mathbb{R})} .$$

 Thus, we show the desired lower estimate for the spaces $B_{2,\infty}^{s}$. By dropping the $L^2$-norm and using Plancherel's we get 
\begin{align*}
\left \| f \left(\frac{\cdot}{\lambda^\delta}\right) \cos(\lambda \cdot - a) \right \|_{B_{2,\infty}^{\sigma}(\mathbb{R})} & \gsim
 \sup_{h \not =0}|h|^{-\sigma} \left \|  \mathcal{F}\bigg[\Delta_{h}^{m} \left( f \left(\frac{\cdot}{\lambda^\delta}\right) \cos(\lambda \cdot - a)\right) \bigg] \right \|_{L^{2}(\mathbb{R})}.
\end{align*}

Elementary properties of the Fourier transform and \eqref{eqn:itdiff} give
\begin{align*} \mathcal{F}\bigg[\Delta_{h}^{m} \left( f \left(\frac{\cdot}{\lambda^\delta}\right) \cos(\lambda \cdot - a)\right) \bigg](\xi) = \sum_{k=0}^{m} (-1)^{m-k} {m \choose k} e^{i\xi h k} \mathcal{F}\bigg[ f \left(\frac{\cdot}{\lambda^\delta}\right) \cos(\lambda \cdot - a)\bigg](\xi) .\end{align*}

So the right hand side of the latter inequality is equivalent to 
\begin{align*} & \sup_{h \not =0}|h|^{-\sigma}\left(  \int_{\mathbb{R}} \left| \sum_{k=0}^{m} (-1)^{m-k} {m \choose k} e^{i\xi h k} \frac{\lambda^\delta}{2} \left( \hat f \left(\lambda^\delta(\xi - \lambda)\right)e^{-ia} + \hat f \left(\lambda^\delta(\xi + \lambda)\right)e^{ia}\right)  \right|^2 \; d\xi \right)^{1/2}\\
& = \sup_{h \not =0}|h|^{-\sigma}\left(  \int_{\mathbb{R}} \left| \sum_{k=0}^{m} (-1)^{m-k} {m \choose k} e^{i\xi h k} \right|^2  \frac{\lambda^{2\delta}}{4} \left| e^{-ia} \hat f \left(\lambda^\delta(\xi - \lambda)\right) + e^{ia} \hat f \left(\lambda^\delta(\xi + \lambda)\right)   \right|^2   \; d\xi \right)^{1/2} . \end{align*}

From translation and dilation changes of variables the latter equals 
\begin{align*}  & \sup_{h \not =0}|h|^{-\sigma}\left(  \int_{\mathbb{R}} \left| \sum_{k=0}^{m} (-1)^{m-k} {m \choose k} e^{i(\xi + \lambda) h k} \right|^2  \frac{\lambda^{2\delta}}{4} \left| e^{-ia} \hat f (\lambda^\delta \xi) + e^{ia} \hat f \left(\lambda^\delta(\xi + 2\lambda)\right)   \right|^2 \; d\xi \right)^{1/2} \\
& = \sup_{h \not =0}|h|^{-\sigma}\left(  \int_{\mathbb{R}} \left| \sum_{k=0}^{m} (-1)^{m-k} {m \choose k} e^{i(\lambda^{-\delta}\xi + \lambda) h k} \right|^2  \frac{\lambda^\delta}{4} \left| e^{-ia} \hat f ( \xi) + e^{ia} \hat f (\xi + 2\lambda^{\delta+1})   \right|^2  \; d\xi \right)^{1/2}.
 \end{align*}

A  change of variables with respect to the $h$ variable gives the equivalent expression
\begin{align*}
& \lambda^{\sigma }\sup_{\lambda h \not =  0}|\lambda h|^{-\sigma}\left(  \int_{\mathbb{R}} \left| \sum_{k=0}^{m} (-1)^{m-k} {m \choose k} e^{i(\lambda^{-\delta-1}\xi + 1) (\lambda h) k} \right|^2 \frac{\lambda^\delta}{4}  \left| e^{-ia} \hat f ( \xi) + e^{ia} \hat f (\xi + 2\lambda^{\delta+1})   \right|^2 \;d\xi  \right)^{1/2}\\ 
& = \frac{\lambda^{\sigma + \delta/2}}{2}\sup_{h \not =0}| h|^{-\sigma}\left(  \int_{\mathbb{R}} \left| \sum_{k=0}^{m} (-1)^{m-k} {m \choose k} e^{i(\lambda^{-\delta-1}\xi + 1)  h k} \right|^2 \left| e^{-ia} \hat f ( \xi) + e^{ia} \hat f (\xi + 2\lambda^{\delta+1})   \right|^2 \;d\xi  \right)^{1/2} .
\end{align*}

Choosing the value $h = \pi/3$ and applying the binomial theorem gives the  lower bound
\begin{align*}
\lambda^{-2\sigma -\delta} & \left \| f \left(\frac{\cdot}{\lambda^\delta}\right) \cos(\lambda \cdot - a) \right \|_{B_{2,\infty}^{\sigma}(\mathbb{R})}^2 \\ & \gsim  \int_{\mathbb{R}} 2^m\left(1 - \cos\left((\lambda^{-\delta-1}\xi + 1) \pi/3 \right) \right)^m \left| e^{-ia} \hat f ( \xi) + e^{ia} \hat f (\xi + 2\lambda^{\delta+1})   \right|^2 \;d\xi.
\end{align*}

From the triangle inequality we obtain the further lower bound
\begin{align*}
\gsim  \int_{\mathbb{R}} 2^m\left(1 - \cos\left((\lambda^{-\delta-1}\xi + 1) \pi/3 \right) \right)^m |\hat f(\xi)|^2 \;d\xi  - \int_{\mathbb{R}} |\hat f(\xi +2\lambda^{\delta+1})|^2 \;d\xi .
\end{align*}

Finally, we use Lebesgue's dominated convergence theorem and take the limit as $\lambda \to \infty$ on the latter expression to obtain the following estimate for $\lambda \gg 1$
\begin{align*}  \lambda^{-2\sigma - \delta}  \left \| f \left(\frac{\cdot}{\lambda^\delta}\right) \cos(\lambda \cdot - a) \right \|_{B_{2,\infty}^{\sigma}(\mathbb{R})}^2 & \gsim   \int_{\mathbb{R}}2^m (1-\cos(\pi/3))^m |\hat f(\xi)|^2 \;d\xi   =   \| f \|_{L^2(\mathbb{R})}^2 .
\end{align*}
\end{proof}
Next we construct the approximate solutions.

\subsection{Approximate Solutions}
The approximate solutions are given by
$$ u^{\omega, \lambda}(t,x) = u^{h}(t,x) + u^{l}(t,x),\;\;\; (t,x)\in \mathbb{R}\times \mathbb{R}^2.$$

The high frequency term, $u^h$, is defined as 
\begin{align*} u^{h}(t,x) & = (\partial_2, -\partial_1) \bigg\{ \lambda^{-\delta -s -1}\phi \left(\frac{x_1}{\lambda^\delta}\right)\phi \left(\frac{x_2}{\lambda^\delta}\right)\sin(\lambda x_2 - \omega t)\bigg\},
\end{align*}
where the function $\phi \in C_{c}^{\infty}(\mathbb{R})$ is supported in $[-2,2]$ and $\phi = 1 $ in $(-1,1)$, $\lambda \in \mathbb{N}$ and the parameters $s \in \mathbb{R} $, $\delta > 0 $ are specified later. The low frequency term, $u^l$, is defined as the solution to the non-local Cauchy problem \eqref{eqn:NLEE}  with corresponding initial data given by
\begin{align*}u^l(0,x) &  = (\partial_2, -\partial_1) \bigg\{ -\omega\lambda^{\delta -1}\psi_1 \left(\frac{x_1}{\lambda^\delta}\right)\psi_2 \left(\frac{x_2}{\lambda^\delta}\right) \bigg\},
\end{align*}
where the localizing functions $\psi_1,\psi_2 \in C_{c}^{\infty}(\mathbb{R})$ are chosen so that $\psi_1 ' = \psi_2 =1$ on the support of $\phi$.

\subsection{Error Terms and Their Estimates}
Plugging in $u^{\omega,\lambda}$ into equation \eqref{eqn:NLEE} gives 
\begin{align*}
\partial_t u^{\omega,\lambda} & + \nabla_{u^{\omega,\lambda}}u^{\omega,\lambda} - \nabla\Delta^{-1}\text{div}(\nabla_{u^{\omega,\lambda}}u^{\omega,\lambda})  \\
& = \partial_t u^{h} + \nabla_{u^{l}}u^{h} + \nabla_{u^{h}}u^{l} + \nabla_{u^{h}}u^{h} - 2\nabla\Delta^{-1}\text{div}(\nabla_{u^{l}}u^{h}) - \nabla\Delta^{-1}\text{div}(\nabla_{u^{h}}u^{h}) \\
& =: \sum_{j=1}^{6}E_j.
\end{align*} 

In order to estimate the error terms, $E_1,\hdots,E_6$, we first use Lemma \ref{tlemnp} and estimate \eqref{eqn:energy} to estimate the high and low frequencies as follows.
\begin{align*} 
\| u^{l}(t,\cdot) \|_{B_{2,q}^{\sigma}(\mathbb{R}^2)} & \lsim  \| u^{l}(0,\cdot) \|_{B_{2,q}^{\sigma}(\mathbb{R}^2)} \simeq \lambda^{-1 + \delta}, \;\; \text{if}\; \sigma >2. \\
\| u^{h}(t,\cdot) \|_{L^{\infty}(\mathbb{R}^2)} \lsim \lambda^{-s -\delta} \;\;\; & \text{and} \;\;\; \|u^{h}(t,\cdot) \|_{B_{2,q}^{\sigma}(\mathbb{R}^2)} \lsim \lambda^{-s + \sigma}.
\end{align*}

We start estimating $E_1 + E_2$ explicitly. \\

We will find it convenient to use the condition $\psi_1 ' = \psi_2 =1$ on the support of $\phi$ to write the first component of $E_1 + E_2$ as follows.
\begin{align*} (E_1 + E_2)_1(t,x) & =  \lambda^{-s+1-\delta}\left( u_{2}^{l}(0,x) - u_{2}^{l}(t,x)\right)\phi\left(\frac{x_1}{\lambda^\delta} \right)\phi\left(\frac{x_2}{\lambda^\delta} \right)\sin(\lambda x_2 - \omega t)\\
& - \omega \lambda^{-s-1-2\delta} \phi\left(\frac{x_1}{\lambda^\delta} \right)\phi'\left(\frac{x_2}{\lambda^\delta} \right)\cos(\lambda x_2 - \omega t) \\
& + \lambda^{-s-2\delta}u_{1}^{l}(t,x) \phi' \left(\frac{x_1}{\lambda^\delta} \right)\phi \left(\frac{x_2}{\lambda^\delta} \right)\cos(\lambda x_2 - \omega t)\\
& + \lambda^{-s-1-3\delta}u_{1}^{l}(t,x) \phi' \left(\frac{x_1}{\lambda^\delta} \right)\phi' \left(\frac{x_2}{\lambda^\delta} \right)\sin(\lambda x_2 - \omega t) \\
& + 2\lambda^{-s-2\delta}u_{2}^{l}(t,x) \phi \left(\frac{x_1}{\lambda^\delta} \right)\phi '\left(\frac{x_2}{\lambda^\delta} \right)\cos(\lambda x_2 - \omega t)\\
& + \lambda^{-s-1-3\delta}u_{2}^{l}(t,x) \phi \left(\frac{x_1}{\lambda^\delta} \right)\phi ''\left(\frac{x_2}{\lambda^\delta} \right)\sin(\lambda x_2 - \omega t).
\end{align*}

Which we bound for $\omega = \pm 1$ and $\sigma > 1$ as follows.
\begin{align*}
\left\| (E_1 + E_2)_1 \right\|_{B_{2,q}^{\sigma}(\mathbb{R}^2)} & \le   \lambda^{-s+1-\delta} \| u_{2}^{l}(0,\cdot) - u_{2}^{l}(t,\cdot)\|_{B_{2,q}^{\sigma}(\mathbb{R}^2)} \left\| \phi\left(\frac{\cdot}{\lambda^\delta} \right)\phi\left(\frac{\cdot}{\lambda^\delta} \right)\sin(\lambda \cdot - \omega t)\right\|_{B_{2,q}^{\sigma}(\mathbb{R}^2)}\\ 
& +  \lambda^{-s-1-2\delta} \left\| \phi\left(\frac{\cdot}{\lambda^\delta} \right) \right\|_{B_{2,q}^{\sigma}(\mathbb{R})} \left\| \phi'\left(\frac{\cdot}{\lambda^\delta} \right)\cos(\lambda \cdot - \omega t) \right\|_{B_{2,q}^{\sigma}(\mathbb{R})}\\
& + \lambda^{-s-2\delta}\|u_{1}^{l}(t,\cdot)\|_{B_{2,q}^{\sigma}(\mathbb{R}^2)} \left\|\phi' \left(\frac{\cdot}{\lambda^\delta} \right)\phi \left(\frac{\cdot}{\lambda^\delta} \right)\cos(\lambda \cdot - \omega t) \right\|_{B_{2,q}^{\sigma}(\mathbb{R}^2)} \\
& + \lambda^{-s-1-3\delta}\|u_{1}^{l}(t,\cdot)\|_{B_{2,q}^{\sigma}(\mathbb{R}^2)} \left\|\phi' \left(\frac{\cdot}{\lambda^\delta} \right)\phi' \left(\frac{\cdot}{\lambda^\delta} \right)\sin(\lambda \cdot - \omega t) \right\|_{B_{2,q}^{\sigma}(\mathbb{R}^2)}  \\
& + 2\lambda^{-s-2\delta}\|u_{2}^{l}(t,\cdot)\|_{B_{2,q}^{\sigma}(\mathbb{R}^2)} \left\|\phi \left(\frac{\cdot}{\lambda^\delta} \right)\phi '\left(\frac{\cdot}{\lambda^\delta} \right)\cos(\lambda \cdot - \omega t)\right\|_{B_{2,q}^{\sigma}(\mathbb{R}^2)} \\
& + \lambda^{-s-1-3\delta}\|u_{2}^{l}(t,\cdot)\|_{B_{2,q}^{\sigma}(\mathbb{R}^2)} \left\|\phi \left(\frac{\cdot}{\lambda^\delta} \right)\phi ''\left(\frac{\cdot}{\lambda^\delta} \right)\sin(\lambda \cdot - \omega t)\right\|_{B_{2,q}^{\sigma}(\mathbb{R}^2)}  \\
& \lsim \lambda^{-s+1+\sigma}\int_{0}^{T}\|\partial_t u^{l}(t,\cdot)\|_{B_{2,q}^{\sigma}(\mathbb{R}^2)}\; dt  +  \lambda^{-s-1-\delta+\sigma}  + \lambda^{-s-1+\sigma}  + \lambda^{-s-2-\delta+\sigma}.
\end{align*}

We use Lemma \ref{tlemnp} and estimates \eqref{eqn:hodge2} and \eqref{eqn:energy} to get
\begin{align*}
\int_{0}^{T}\|\partial_t u^{l}\|_{B_{2,q}^{\sigma}(\mathbb{R}^2)}\; dt  & = \int_{0}^{T}\|P(\nabla_{u^{l}} u^{l})\|_{B_{2,q}^{\sigma}(\mathbb{R}^2)}\; dt  \le \int_{0}^{T}\|\nabla_{u^{l}} u^{l}\|_{B_{2,q}^{\sigma}(\mathbb{R}^2)}\; dt \\ 
& \lsim \int_{0}^{T} \| u^{l}\|_{B_{2,q}^{\sigma}(\mathbb{R}^2)} \|  u^{l}\|_{B_{2,q}^{\sigma+1}(\mathbb{R}^2)}\; dt  \lsim (\lambda^{-1+\delta})^2 ,\end{align*}
from where we deduce the lower bound
\begin{align*}
\left\| (E_1 + E_2)_1 \right\|_{B_{2,q}^{\sigma}(\mathbb{R}^2)} & \lsim \lambda^{-s+1+\sigma}(\lambda^{-1+\delta})^2   +  \lambda^{-s-1-\delta+\sigma}  + \lambda^{-s-1+\sigma}  + \lambda^{-s-2-\delta+\sigma} \lsim \lambda^{-s-1+2\delta+\sigma} .
\end{align*}

The second component of $E_1 + E_2$ is given by
\begin{align*}
(E_1 + E_2)_2(t,x) & =  \omega \lambda^{-s-1-2\delta} \phi' \left(\frac{x_1}{\lambda^\delta} \right)\phi\left(\frac{x_2}{\lambda^\delta} \right)\cos(\lambda x_2 - \omega t)\\
&  -  \lambda^{-s-1-3\delta} u_{1}^{l}(t,x) \phi'' \left(\frac{x_1}{\lambda^\delta} \right)\phi \left(\frac{x_2}{\lambda^\delta} \right)\sin(\lambda x_2 - \omega t)\\
& - \lambda^{-s-2\delta}u_{2}^{l}(t,x) \phi' \left(\frac{x_1}{\lambda^\delta} \right)\phi \left(\frac{x_2}{\lambda^\delta} \right)\cos(\lambda x_2 - \omega t) \\
& - \lambda^{-s-1-3\delta}u_{2}^{l}(t,x) \phi' \left(\frac{x_1}{\lambda^\delta} \right)\phi' \left(\frac{x_2}{\lambda^\delta} \right)\sin(\lambda x_2 - \omega t).
\end{align*}

Which we similarly bound as follows.
\begin{align*}
\left\|(E_1 + E_2)_2\right\|_{B_{2,q}^{\sigma}(\mathbb{R}^2)} & \le  \lambda^{-s-1-2\delta} \left\| \phi' \left(\frac{\cdot}{\lambda^\delta} \right)\right\|_{B_{2,q}^{\sigma}(\mathbb{R})} \left\|\ \phi\left(\frac{\cdot}{\lambda^\delta} \right)\cos(\lambda \cdot - \omega t) \right\|_{B_{2,q}^{\sigma}(\mathbb{R})}\\
& +  \lambda^{-s-1-3\delta} \|u_{1}^{l}(t,\cdot)\|_{B_{2,q}^{\sigma}(\mathbb{R}^2)} \left\| \phi'' \left(\frac{\cdot}{\lambda^\delta} \right)\phi \left(\frac{\cdot}{\lambda^\delta} \right)\sin(\lambda \cdot - \omega t)\right\|_{B_{2,q}^{\sigma}(\mathbb{R}^2)} \\
& + \lambda^{-s-2\delta} \|u_{2}^{l}(t,\cdot)\|_{B_{2,q}^{\sigma}(\mathbb{R}^2)} \left\| \phi' \left(\frac{\cdot}{\lambda^\delta} \right)\phi \left(\frac{\cdot}{\lambda^\delta} \right)\cos(\lambda \cdot - \omega t)\right\|_{B_{2,q}^{\sigma}(\mathbb{R}^2)}\\
& + \lambda^{-s-1-3\delta}\|u_{2}^{l}(t,\cdot)\|_{B_{2,q}^{\sigma}(\mathbb{R}^2)} \left\| \phi' \left(\frac{\cdot}{\lambda^\delta} \right)\phi' \left(\frac{\cdot}{\lambda^\delta} \right)\sin(\lambda \cdot - \omega t)\right\|_{B_{2,q}^{\sigma}(\mathbb{R}^2)} \\
& \lsim \lambda^{-s-1-\delta + \sigma} + \lambda^{-s-2-\delta+\sigma} + \lambda^{-s-1+\sigma}.
\end{align*}

Combining the estimates above we obtain
\begin{align*} \left\|E_1 + E_2\right\|_{B_{2,q}^{\sigma}(\mathbb{R}^2)} & \lsim \lambda^{-s-1+2\delta+\sigma} +  \lambda^{-s-1-\delta+\sigma}  + \lambda^{-s-1+\sigma}  + \lambda^{-s-2-\delta+\sigma} \lsim \lambda^{-s-1+2\delta+\sigma} .
\end{align*}

We use Lemma \ref{tlemnp}, estimate \eqref{eqn:hodge2}  and Sobolev-type embeddings to get
\begin{align*}
\|E_3 + E_5 \|_{B_{2,q}^{\sigma}(\mathbb{R}^2)} & = \left\|P(\nabla_{u^h}u^l) - (1-P)(\nabla_{u^h}u^l) \right\|_{B_{2,q}^{\sigma}(\mathbb{R}^2)}  \lsim \left\| \nabla_{u^h}u^l \right\|_{ B_{2,q}^{\sigma}(\mathbb{R}^2) } \\
& \lsim \|u^h \|_{ B_{2,q}^{\sigma}(\mathbb{R}^2) } \| u^{l}\|_{B_{2,q}^{\sigma+1}(\mathbb{R}^2)}  \lsim  \lambda^{-s+\sigma-1+\delta}.
\end{align*}

From the definition of $u^h$ and the estimate \eqref{eqn:moser} with $p_1 = p_3 = \infty$, $p_2 = p_4 = 2$, we obtain
\begin{align*}
\|E_4 + E_6 \|_{B_{2,q}^{\sigma}(\mathbb{R}^2)} & = \left\|P(\nabla_{u^h}u^h) \right\|_{B_{2,q}^{\sigma}(\mathbb{R}^2)} \le \left\| \nabla_{u^h}u^h \right\|_{B_{2,q}^{\sigma}(\mathbb{R}^2)}
\lsim \lambda^{-2s-2\delta + \sigma}.
\end{align*}

We conclude that for any $\sigma >1$ 
\begin{equation}\label{eqn:errest}
\Big\| \sum_{j=1}^{6} E_j \Big\|_{B_{2,q}^{\sigma}(\mathbb{R}^2)}  \lsim \lambda^{-(s+1-2\delta-\sigma)}+ \lambda^{-(s-\sigma+1-\delta)} + \lambda^{-(2s+2\delta-\sigma)} =: d_1(\lambda).
\end{equation}

We are in place to construct the exact solutions to \eqref{eqn:NLEE}.
\subsection{Exact Solutions}
We let $u_{\omega,\lambda}$ be the solution to equation \eqref{eqn:NLEE} with initial data $u_{\omega,\lambda}(0) = u^{\omega,\lambda}(0)$.  Also, let $v = u^{\omega,\lambda} - u_{\omega,\lambda}$ be the difference between the approximation and exact solution, respectively, and observe that $v$ solves the transport equation  
\begin{equation}\label{eqn:tran.sol} \partial_t v + (u^{\omega,\lambda} \cdot \nabla) v  = F \;\;\; \text{and} \;\;\; v(0) =0,
\end{equation}
where $F:= \displaystyle\sum_{j=1}^{6}E_j + P(\nabla_v v)   + (1- 2P)(\nabla_{v}u^{\omega,\lambda})$. 

We use estimates for solutions to a linear transport equation as convenient here. 
\begin{lemma}[See \cite{BCD}]\label{trans} Let $d/p < \sigma <1+d/p$, $1\le p, q \le \infty$ and assume that div $\mu =0$. There exists a constant $C=C(p,q,\sigma)$ so that for all solutions $f \in L^\infty ([0,T]; B_{p,q}^{s})$ of
\begin{align*}
\partial_t f+ (\mu \cdot \nabla) f& = F, \\
f(0) & = f_0,
\end{align*}
with $f_0 \in B_{p,q}^{s}$ and $F \in L^1\left( [0,T]; B_{p,q}^{s}\right)$,  we have for a.e. $t\in[0,T]$ 
$$\|f\|_{L^\infty_t(
B_{p,q}^{s})} \le e^{CV(t)}\left( \|f_0\|_{B_{p,q}^{s}} + \int_{0}^{t} \|F(\tau)\|_{B_{p,q}^{s}} \; d\tau\right),$$
where $V(t) = \displaystyle \int_{0}^{t}\|\nabla \mu (\tau)\|_{B_{p,\infty}^{d/p}\cap L^\infty}\; d\tau$.
\end{lemma}

We apply Lemma \ref{trans} to $v = u^{\omega,\lambda} - u_{\omega,\lambda}$ with $1 < \sigma <  \min\{2,s-1\} $  and obtain the estimate 
\begin{align*}
\|v\|_{B_{2,q}^{\sigma}(\mathbb{R}^2)} \le C e^{CV(t)}\int_{0}^{t}\|F(\tau)\|_{B_{2,q}^{\sigma}(\mathbb{R}^2)}\; d\tau .
\end{align*}

For any $t \in[0,1]$ Lemma \ref{tlemnp} and the Sobolev-type embedding give
\begin{align*} 
V(t)  & = \int_{0}^{t}\| \nabla u^{\omega,\lambda}(\tau)\|_{B_{2,\infty}^{1}\cap L^\infty(\mathbb{R}^2)} \; d\tau  \lsim \int_{0}^{t}\| u^{\omega,\lambda}(\tau)\|_{B_{2,\infty}^{\sigma+1}(\mathbb{R}^2)}\; d\tau \lsim \lambda^{-1+\delta} + \lambda^{-s+\sigma+1}.
\end{align*}

Using the triangle inequality together with estimates \eqref{eqn:hodge2} and \eqref{eqn:errest} we obtain 
$$\|F \|_{B_{2,q}^{\sigma}(\mathbb{R}^2)} \lsim d_1(\lambda) +  \| \nabla_v v\|_{B_{2,q}^{\sigma}(\mathbb{R}^2)}  + \| \nabla _v u^{\omega,\lambda}\|_{B_{2,q}^{\sigma}(\mathbb{R}^2)}.  $$

Using the estimates for the high and low frequencies we obtain
\begin{align*} 
\| \nabla_v v\|_{B_{2,q}^{\sigma}(\mathbb{R}^2)} & = \| (v\cdot\nabla) v\|_{B_{2,q}^{\sigma}(\mathbb{R}^2)}  \le  \sum_{j=1}^{2} \| v_j \partial_jv\|_{B_{2,q}^{\sigma}(\mathbb{R}^2)}  \lsim \sum_{j=1}^{2} \|  v_j\|_{B_{2,q}^{\sigma}(\mathbb{R}^2)}\| \partial_j v \|_{B_{2,q}^{\sigma}(\mathbb{R}^2)} \\ &  \lsim \| v\|_{B_{2,q}^{\sigma}(\mathbb{R}^2)} \left( \lambda^{-1+\delta} + \lambda^{-s+\sigma+1} \right).
\end{align*}

Similarly, 
\begin{align*} 
\| \nabla_{v} u^{\omega,\lambda}\|_{B_{2,q}^{\sigma}(\mathbb{R}^2)}  & = \| (v \cdot\nabla) u^{\omega,\lambda}\|_{B_{2,q}^{\sigma}(\mathbb{R}^2)} \le  \sum_{j=1}^{2} \| v_j \partial_ju^{\omega,\lambda}\|_{B_{2,q}^{\sigma}(\mathbb{R}^2)}  \lsim \sum_{j=1}^{2} \|  v_j\|_{B_{2,q}^{\sigma}(\mathbb{R}^2)}\| \partial_j u^{\omega,\lambda} \|_{B_{2,q}^{\sigma}(\mathbb{R}^2)} \\ &  \lsim \| v\|_{B_{2,q}^{\sigma}(\mathbb{R}^2)} \left( \lambda^{-1+\delta} + \lambda^{-s+\sigma+1} \right),
\end{align*}
so that collecting all these estimates we have
\begin{align*}
\|v\|_{B_{2,q}^{\sigma}(\mathbb{R}^2)} & \lsim e^{CV(t)}\int_{0}^{t}\|F(\tau)\|_{B_{2,q}^{\sigma}(\mathbb{R}^2)}\; d\tau \lsim  d_1(\lambda) + \int_{0}^{t} \left[ \lambda^{-(1-\delta)} + 1  \right]  \|  v \|_{B_{2,q}^{\sigma}(\mathbb{R}^2)} \;ds.
\end{align*}

Finally, an application of Gr\"onwall's inequality gives the desired estimate
\begin{align*} 
\|v\|_{B_{2,q}^{\sigma}(\mathbb{R}^2)} & \lsim d_1(\lambda)\int_{0}^{t}e^{(t-s)(\lambda^{-(1-\delta)} +1)} \; ds  \lsim d_1(\lambda), \end{align*}
uniformly for $t\in[0,1]$ and $\lambda \gg 1$ if $0< \delta <1$. We are now ready to prove Theorem \ref{main3}.

\subsection{Proof of Theorem \ref{main3}}
Let  $u_{+1,\lambda}(t)$ and $u_{-1,\lambda}(t)$ be two sequences of solutions to \eqref{eqn:NLEE} with corresponding initial data $u_{+1,\lambda}(0) = u^{+1,\lambda}(0)$ and $u_{-1,\lambda}(0) = u^{-1,\lambda}(0)$. Estimate \eqref{eqn:energy} gives
\begin{align*}
\| u_{\pm1,\lambda} (t)\|_{B_{2,q}^{s}(\mathbb{R}^2)} & \le \| u_{\pm1,\lambda} (0)\|_{B_{2,q}^{s}(\mathbb{R}^2)}  =    \|u^{\pm1,\lambda} (0)\|_{B_{2,q}^{s}(\mathbb{R}^2)} \\  & \lsim \| u^{l} (0)\|_{B_{2,q}^{s}(\mathbb{R}^2)} + \| u^{h} (0)\|_{B_{2,q}^{s}(\mathbb{R}^2)} \\
& \lsim \lambda^{-1+\delta} + \lambda^{-s+s} \\
& \lsim 1.
\end{align*}
for any $0 < \delta < 1$. The latter says that our solutions are confined to a ball in $B_{2,q}^{s}$, this is (i).\\

We use interpolation to measure the difference between the approximate and exact solutions and thus need the following estimates. For $0 < \delta < 1$ and any integer $k>\max\{2,s\}$ we obtain
\begin{align*} \| u^{\pm1,\lambda}(t) \|_{B_{2,q}^{k}(\mathbb{R}^2)} & \le \| u^{l}(t)\|_{B_{2,q}^{k}(\mathbb{R}^2)}  + \| u^{h}(t)\|_{B_{2,q}^{k}(\mathbb{R}^2)}  \lsim \lambda^{-1+\delta} + \lambda^{-s+k} \lsim \lambda^{-s+k}
\end{align*}

uniformly in $t\in [0,1]$, where we used Lemma \ref{tlemnp}. Repeating this estimate twice gives 
\begin{equation}\label{eqn:distapack} \| u^{\pm1,\lambda}(t) - u_{\pm1,\lambda}(t)\|_{B_{2,q}^{k}(\mathbb{R}^2)} \lsim \lambda^{-s+k}\end{equation}

uniformly in $t\in [0,1]$. Expanding $v$ and $d_1(\lambda)$ as defined in the previous subsection with $\omega = \pm 1$, $0<\delta<1$ and $1 < \sigma < \min\{2,s-1\}$ we see that
\begin{equation}\label{eqn:distapacsig}\|u^{\pm1,\lambda}(t) - u_{\pm1,\lambda}(t) \|_{B_{2,q}^{\sigma}(\mathbb{R}^2)} \lsim \lambda^{-(s+1-2\delta-\sigma)}+ \lambda^{-(s-\sigma+1-\delta)} + \lambda^{-(2s+2\delta-\sigma)}. \end{equation}

In what follows fix $s > 2$, $0 < \delta <1/2$, and choose $1 < \sigma < \min\{2,s-1\}$. \\

We let $s_1 := \sigma $ and $s_2 := k > s$, and define $\theta \in (0,1)$ by $s = \theta \sigma + (1-\theta)k$. Interpolation inequality \eqref{eqn:interp} and the estimates \eqref{eqn:distapack} and \eqref{eqn:distapacsig} gives
\begin{equation}\label{eqn:finaldiff} 
\begin{aligned}
\|u^{\pm1,\lambda}(t) - u_{\pm1,\lambda}(t) \|_{B_{2,q}^{s}(\mathbb{R}^2)} & \lsim \|u^{\pm1,\lambda}(t) - u_{\pm1,\lambda}(t) \|_{B_{2,q}^{\sigma}(\mathbb{R}^2)}^\theta \; \|u^{\pm1,\lambda}(t) - u_{\pm1,\lambda}(t) \|_{B_{2,q}^{k}(\mathbb{R}^2)}^{1-\theta}\\ 
& \lsim \left[ \lambda^{-(s+1-2\delta-\sigma)}+ \lambda^{-(s-\sigma+1-\delta)} + \lambda^{-(2s+2\delta-\sigma)} \right]^{\frac{k - s}{k - \sigma}}\left( \lambda^{-s+k} \right)^{\frac{s-\sigma}{k-\sigma}}\\
&= \bigg( \lambda^{-(1-2\delta)}+ \lambda^{-(1-\delta)} + \lambda^{-(s+2\delta)}\bigg)^{\frac{k - s}{k - \sigma}} \\ 
& =: d_2(\lambda) \to 0 \; \text{as} \; \lambda \to \infty. 
\end{aligned}
\end{equation}

Next we make sure that the exact solutions converge at the initial time, this is (ii). Using $u_{\pm1,\lambda}(0) := u^{\pm1,\lambda}(0) $, estimate \eqref{eqn:energy} and Lemma \ref{tlemnp} we get 
\begin{align*}
\| u_{+1,\lambda}(0) - u_{-1,\lambda}(0)\|_{B_{2,q}^{s}(\mathbb{R}^2)}
& \lsim \| u^{+1,l}(0)\|_{B_{2,q}^{s}(\mathbb{R}^2)} + \| u^{-1,l}(0)\|_{B_{2,q}^{s}(\mathbb{R}^2)}\\ 
& \lsim \lambda^{-1} \| \psi_1 \left(\frac{\cdot}{\lambda^\delta}\right)  \|_{B_{2,q}^{s}(\mathbb{R})} \|  \psi_2' \left(\frac{\cdot}{\lambda^\delta}\right)  \|_{B_{2,q}^{s}(\mathbb{R})}  \\
& + \lambda^{-1} \| \psi_1' \left(\frac{\cdot}{\lambda^\delta}\right)  \|_{B_{2,q}^{s}(\mathbb{R})} \|  \psi_2 \left(\frac{\cdot}{\lambda^\delta}\right)  \|_{B_{2,q}^{s}(\mathbb{R})} \\ 
& \lsim \lambda^{-1+\delta} \rightarrow 0\;\; \text{as} \;\; \lambda \to \infty.
\end{align*}

Finally, we make sure that the exact solutions are separated at later times, this is (iii). Let $t \in (0,1]$, using the triangle inequality and the estimate \eqref{eqn:finaldiff} we see that the key estimate for separating the exact solutions is given by bounding away the approximate solutions, i.e.
\begin{align*}
\| u_{+1,\lambda}(t) - u_{-1,\lambda}(t)\|_{B_{2,q}^{s}(\mathbb{R}^2)} & \ge \| u^{+1,\lambda}(t) - u^{-1,\lambda}(t)\|_{B_{2,q}^{s}(\mathbb{R}^2)}  - \| u^{+1,\lambda}(t) - u_{+1,\lambda}(t)\|_{B_{2,q}^{s}(\mathbb{R}^2)} \\
& - \| u^{-1,\lambda}(t) - u_{-1,\lambda}(t)\|_{B_{2,q}^{s}(\mathbb{R}^2)} \\ 
& \gsim \| u^{+1,\lambda}(t) - u^{-1,\lambda}(t)\|_{B_{2,q}^{s}(\mathbb{R}^2)} - d_2(\lambda).
\end{align*}

The definition of the approximate solution and the triangle inequality gives
\begin{align*}
\| u^{+1,\lambda}(t,\cdot) - u^{-1,\lambda}(t,\cdot)\|_{B_{2,q}^{s}(\mathbb{R}^2)} & \ge \| u^{+1,h}(t,\cdot) - u^{-1,h}(t,\cdot)\|_{B_{2,q}^{s}(\mathbb{R}^2)} \\
& -\| u^{+1,l}(t,\cdot)\|_{B_{2,q}^{s}(\mathbb{R}^2)} - \| u^{-1,l}(t,\cdot)\|_{B_{2,q}^{s}(\mathbb{R}^2)} \\
& \gsim \| u^{+1,h}(t,\cdot) - u^{-1,h}(t,\cdot)\|_{B_{2,q}^{s}(\mathbb{R}^2)} - \lambda^{-1+\delta}.
\end{align*}

Using the definition of $u^{\pm1,h}$ and the angle sum formula we write $u^{+1,h}(t,x) - u^{-1,h}(t,x)$ as follows
\begin{align*}
 2\sin t \; & \bigg( \lambda^{-\delta -s }\phi \left(\frac{x_1}{\lambda^\delta}\right) \phi \left(\frac{x_2}{\lambda^\delta}\right) \sin(\lambda x_2)  - \lambda^{-2\delta -s -1}\phi \left(\frac{x_1}{\lambda^\delta}\right)\phi ' \left(\frac{x_2}{\lambda^\delta}\right)
\cos(\lambda x_2),\\
&  \lambda^{-2\delta -s -1}\phi ' \left(\frac{x_1}{\lambda^\delta}\right)\phi \left(\frac{x_2}{\lambda^\delta}\right)\cos(\lambda x_2) \bigg).
\end{align*}

Applying Lemma \ref{tlemnp} and the triangle inequality gives
\begin{align*} 
\| u^{+1,h}(t,\cdot) - u^{-1,h}(t,\cdot)\|_{B_{2,q}^{s}(\mathbb{R}^2)} 
&  \gsim \lambda^{-\delta-s} \sin t \left\|  \phi \left(\frac{\cdot}{\lambda^\delta}\right)  \phi \left(\frac{\cdot}{\lambda^\delta}\right)\sin (\lambda \cdot) \right\|_{B_{2,q}^{s}(\mathbb{R}^2)} \\
&- \lambda^{-2\delta-s-1} \sin t \left\|  \phi \left(\frac{\cdot}{\lambda^\delta}\right)  \phi' \left(\frac{\cdot}{\lambda^\delta}\right)\cos (\lambda \cdot) \right\|_{B_{2,q}^{s}(\mathbb{R}^2)}\\
& - \lambda^{-2\delta-s-1} \sin t \left\|  \phi' \left(\frac{\cdot}{\lambda^\delta}\right) \phi \left(\frac{\cdot}{\lambda^\delta}\right)\cos (\lambda \cdot)\right\|_{B_{2,q}^{s}(\mathbb{R}^2)}\\
& \gsim  \lambda^{-\delta-s} \sin t \cdot \lambda^{s+\delta}  - \lambda^{-2\delta-s-1} \cdot \lambda^{\delta/2} \cdot \lambda^{s+\delta/2}.
\end{align*}

Collecting all the estimates above gives, 
\begin{align*}
\| u_{+1,\lambda}(t,\cdot) - u_{-1,\lambda}(t,\cdot)\|_{B_{2,q}^{s}(\mathbb{R}^2)} &\gsim \| u^{+1,\lambda}(t,\cdot) - u^{-1,\lambda}(t,\cdot)\|_{B_{2,q}^{s}(\mathbb{R}^2)} -d_2(\lambda) \\
& \gsim \| u^{+1,h}(t,\cdot) - u^{-1,h}(t,\cdot)\|_{B_{2,q}^{s}(\mathbb{R}^2)} - \lambda^{-1+\delta}- d_2(\lambda)\\
& \gsim \sin t  - \lambda^{-\delta-1} - \lambda^{-1+\delta} - d_2(\lambda).
\end{align*}

From where we conclude
$$ \liminf_{\lambda \to \infty} \|u_{+1,\lambda}(t,\cdot) - u_{-1,\lambda}(t,\cdot)\|_{B_{2,q}^{s}(\mathbb{R}^2)} \gsim \sin t  ,\;\;\;\; \forall t \in (0,1].$$

This completes the proof in the case $d=2$.

For the case $d=3$ the same argument provides the result for $s > 1 + 3/2$ upon taking the high and low frequencies to be 
$$ u^{h}(t,x) = (\partial_2,-\partial_1,0)\lambda^{-3\delta/2-s-1}\prod_{i\in[3]}\phi\left(\frac{x_i}{\lambda^{\delta}}\right)\sin(\lambda x_2 - \omega t)$$
and 
$$ u^{l}(t,x) = (-\omega \partial_2,\omega\partial_1,0)\; \lambda^{-1+\delta}\prod_{i\in[3]}\psi_i\left(\frac{x_i}{\lambda^{\delta}}\right).$$

Where $\phi$ and the $\psi_i$ would be appropriately defined as before with the condition $\psi_1' =  \psi_2 = \psi_3 =1$ on the support of $\phi$. 


\end{document}